\documentclass[10pt]{article}
\font\smallit=cmti10

\usepackage{amssymb,latexsym,amsmath,epsfig,amsthm} 
\usepackage{url}

\makeatletter

\renewcommand\section{\@startsection {section}{1}{\z@}
{-30pt \@plus -1ex \@minus -.2ex}
{2.3ex \@plus.2ex}
{\normalfont\normalsize\bfseries}}

\renewcommand\subsection{\@startsection{subsection}{2}{\z@}
{-3.25ex\@plus -1ex \@minus -.2ex}
{1.5ex \@plus .2ex}
{\normalfont\normalsize\bfseries}}

\renewcommand{\@seccntformat}[1]{\csname the#1\endcsname. }

\makeatother

\newtheorem{theorem}{Theorem}

\newtheorem{corollary}{Corollary}

\newenvironment{example}[1][Example.]{\begin{trivlist}
\item[\hskip \labelsep {\bfseries #1}]}{\end{trivlist}}
\newenvironment{remark}[1][Remark.]{\begin{trivlist}
\item[\hskip \labelsep {\bfseries #1}]}{\end{trivlist}}

\begin{document}

\begin{center}
\uppercase{\bf Derived Palintiple Families and Their Palinomials}
\vskip 20pt
{\bf Benjamin V. Holt}\\
{\smallit Department of Mathematics, Humboldt State University, Arcata, California}\\
{\tt bvh6@humboldt.edu}\\
\vskip 10pt
\end{center}
\vskip 30pt

\vskip 30pt

\centerline{\bf Abstract}

\noindent
We consider several families of palintiples (also known as reverse multiples) whose carries themselves are digits
of lower-base palintiples and give some methods for constructing them from fundamental palintiple types.
We also continue the study of palinomials introduced in an earlier paper by
revealing a more direct relationship between the digits of certain palintiple types and 
the roots of their palinomials. We explore the consequences of this relationship for palinomials induced
by palintiple families derived from lower-base palintiples.
Finally, we pose some questions regarding Young graphs of derived palintiple families 
and consider the implications our general observations might have for relations between Young graph isomorphism classes.

\pagestyle{myheadings}
\thispagestyle{empty}
\baselineskip=12.875pt
\vskip 30pt

\section{Introduction}
In a previous paper on palintiple numbers \cite{holt} (also known as reverse multiples \cite{kendrick_1,sloane,young_1,young_2}) 
it is noted that ``the carries [of a palintiple]...play as critical a role as the digits themselves.''
Indeed, the full measure of this statement is realized when one notices that the carries 
of a palintiple are often themselves the digits of a palintiple of a lower base. 
Consider the example of the $(10,139)$-palintiple $(28, 25, 108, 113, 2)_{139}$ which has carries given by 
$(c_4, c_3, c_2, c_1,c_0)=(8, 7, 1, 2, 0)$. One immediately notices that the nontrivial carries 
are digits of the well-known $(4,10)$-palintiple 8712. 

The recent work of Sloane \cite{sloane} translates the palintiple problem
into graph-theoretical language by means of \textit{Young graphs}, 
which are a succinct visualization of palintiple structure,
showing how the possible carries generate the possible digits of a palintiple of arbitrary length. 
Young graphs are a modification of tree graphs introduced by Young \cite{young_1,young_2} which are a representation
of an efficient palintiple search method with the possible carries represented as nodes 
and the potential digits being associated with the edges.
The full definition of a Young graph can be found in Definition 5 of \cite{kendrick_1}.

We note that Hoey \cite{hoey_1,hoey_2} presented a similar idea using finite state machines.
Representations of machines which recognize palintiples bear strong resemblance to Young graphs;
the Young graph representing $(5,8)$-palintiples in Figure 6 of Sloane's paper \cite{sloane} 
looks very much like the machine which recognizes $(5,8)$-palintiples \cite{hoey_2}.

Kendrick \cite{kendrick_1} extends Sloane's work \cite{sloane} by proving several of his conjectures. 
Most notably, Kendrick \cite{kendrick_1} proves two of Sloane's main conjectures: Theorem 14 
shows that the $(n,b)$-Young graph, $Y(n,b)$, is isomorphic to the ``1089 graph'' 
if and only if $n+1$ divides $b$, and Theorem 31 characterizes \textit{complete} Young graphs. 
(That is, the nontrivial carry-nodes form a subgraph isomorphic to the complete directed graph on $m$ nodes, $K_m$, 
with some additional details which can be found in Definition 3.3 of Sloane \cite{sloane}. 
We note that complete Young graphs are also denoted by $K_m$.)
Kendrick goes on to list several conjectures of his own regarding other Young graph isomorphisms.  
At this point, we inform the reader that Theorem 14 of Kendrick \cite{kendrick_1} 
concerning 1089 graphs will be used repeatedly throughout this paper. 

We also note that the notation we use above to denote the Young graph is slightly different from that of Sloane and Kendrick,
who use $Y(g,k)$, where $g$ is the base, and $k$ is the multiplier (the order of the base and multiplier is reversed).

Other recent work includes \cite{holt} which establishes some general properties of palintiples 
of any base, having an arbitrary number of digits, using only elementary methods. 
As with the work of \cite{kendrick_1,sloane, young_1, young_2}, 
the methods therein pay particular attention to the carries. Patterns found in the carries naturally partition all 
palintiples into three mutually exclusive and exhaustive classes. 
Letting $p=(d_k, d_{k-1},\ldots, d_0)_b$ be an $(n,b)$-palintiple with carries $c_k$,
$c_{k-1}$,$\ldots$, $c_0$, these classes are defined as follows: we say that $p$ is \textit{symmetric} if
$c_j=c_{k-j}$ for all $0\leq j \leq k$, 
and $p$ is \textit{shifted-symmetric} if $c_j=c_{k-j+1}$ for all $0 \leq j \leq k$.
A palintiple which is neither symmetric nor shifted-symmetric is called \textit {asymmetric}.
The $(4,10)$-palintiple $(8,7,1,2)_{10}$ has carries $(c_3,c_2,c_1, c_0)=(0,3,3,0)$, making it an
example of a symmetric palintiple. The reader may find more examples in Table 1 of \cite{holt}.

Comparing the above mentioned classes to Young graph isomorphism classes, 
Theorem 14 in \cite{kendrick_1} and Theorem 6 in \cite{holt}
demonstrate that any $(n,b)$-palintiple generated from a Young graph, $Y(n,b)$, 
which is isomorphic to $Y(9,10)$, otherwise known
as a 1089 graph \cite{kendrick_1, sloane}, is symmetric.
(1089 is a base-10 reverse multiple whose digits are reversed when multiplied by 9:
 $9801=9\cdot 1089$. In the language of this paper, 9801 is a $(9,10)$-palintiple.)
Whether or not every symmetric palintiple can be generated from a 1089 graph
remains an open question (see the last section for a discussion of this).
Also, Theorem 31 of Kendrick \cite{kendrick_1} and Theorem 9 of \cite{holt}
demonstrate that a palintiple is shifted-symmetric if and only if it is generated by a complete Young graph. 
Shifted-symmetric palintiples are the most well-understood as they are completely determined and characterized
by the above works. They are also in some sense a primordial class; all two-digit palintiples
are shifted-symmetric and are a focus of Sutcliffe's \cite{sutcliffe} seminal paper on the topic.

As for palintiples whose Young graph is neither a 1089 graph (symmetric) nor complete (shifted-symmetric),
the asymmetric class is revealed to consist of an astonishing plurality of Young graph isomorphism 
classes (i.e., palintiple types) which admit many subclassifications and isolated cases 
as demonstrated by Kendrick's isomorphism class data \cite{kendrick_2} for all bases less than 337.
Moreover, this plurality seems to only grow with increasing base as suggested by 
Conjecture 43 of Kendrick \cite{kendrick_1}.
This underscores a primary aim of this paper to begin to more fully understand
palintiples beyond the symmetric and shifted-symmetric classes.
We shall describe several families of asymmetric palintiples which are constructed, or \textit{derived}, from lower-base
examples. In particular, we will outline some methods for constructing new asymmetric palintiples, 
such as the example already given above, using as carries 
the digits of ``old'' palintiples whose Young graph is isomorphic to either a 1089 graph (symmetric palintiple)
or a complete graph (shifted-symmetric palintiple).
It is also worth mentioning that, incidentally, three of the four asymmetric examples given in Table 1 of \cite{holt}, 
namely, the $(14,9)$, $(22,7)$, and $(11,7)$-palintiples given by
$(11, 9, 1, 4, 1)_{14}$, $(16, 13, 3, 8, 2)_{22}$, and $(8, 9, 10, 2, 1)_{11}$, respectively, 
are also examples of palintiples \text{derived} from lower-base palintiples.

Kendrick \cite{kendrick_1} mentions that it is still quite poorly understood 
how the number and graph-theoretical aspects of the palintiple problem relate to one another.
Our work, which relies mostly upon elementary results, 
gives rise to some concrete questions as to how the derived palintiple families described here can be classified according to 
Young graph isomorphisms, as well as suggest how Young graph isomorphism classes might 
be generated from others.

Additionally, this paper further develops the topic of palinomials introduced in \cite{holt},
revealing a more intimate relationship between the digits of $(n,b)$-palintiples
and the roots of their palinomials when $Y(n,b)$ is either 1089 or complete. 
These results have implications for palinomials induced by palintiples 
derived from $(n,b)$-palintiples such that $Y(n,b)$ is a 1089 graph or a complete graph.
   
\section{Palintiples Whose Carries are Digits of Lower-Base Palintiples}

Henceforth, we shall suppose that $p=(d_k,d_{k-1}, \ldots, d_0)_b$ is an $(n,b)$-palintiple with carries 
$c_k, c_{k-1},\ldots,c_0$. It is well-established \cite{holt, sloane, young_1} how the 
digits of a palintiple are related to the carries:
\begin{equation}
d_j=\frac{nbc_{k-j+1}-nc_{k-j}+bc_{j+1}-c_j}{n^2-1}.
\label{fund}
\end{equation}

We pose the general question of when the carries of a palintiple are the digits 
of a palintiple of a lower base as in the example given in the introduction. 
In this paper we shall consider two possibilities for when this occurs.
\\\\
\noindent Case 1: We find conditions under which we can construct a new $(k+2)$-digit 
$(\hat{n},\hat{b})$-palintiple $\hat{p}$ with carries $(\hat{c}_{k+1}, \hat{c}_{k},\ldots,\hat{c}_0)$ 
given by $(d_k,d_{k-1}, \ldots, d_0,0)$ as in the example given in the introduction.
Using Equation \ref{fund}, the new digits $\hat{d}_j$ must satisfy
\begin{equation}
\hat{d}_j=\frac{\hat{n}\hat{b}\hat{c}_{k-j+2}-\hat{n}\hat{c}_{k-j+1}+\hat{b} \hat{c}_{j+1}-\hat{c}_{j}}{\hat{n}^2-1}
=\frac{\hat{n}\hat{b}d_{k-j+1}-\hat{n}d_{k-j}+\hat{b}d_j-d_{j-1}}{\hat{n}^2-1}.
\label{fund_rel}
\end{equation}
Then $\hat{b}d_0 \equiv \hat{n} d_k \mod(\hat{n}^2-1)$ when $j=0$. 
Therefore, in order to find a suitable higher base $\hat{b}$, it must be that
$\gcd(d_0,\hat{n}^2-1)$ divides $d_k$, in which case we have that $\hat{b}=s+\alpha \frac{\hat{n}^2-1}{\gcd(d_0,\hat{n}^2-1)}$,
where $s$ is the least non-negative solution of the above congruence and $\alpha \geq 1$. The above then becomes
\begin{equation}
 \hat{d}_j=\frac{\hat{n}s d_{k-j+1}-\hat{n} d_{k-j}+s d_{j}-d_{j-1}}{\hat{n}^2-1}
 +\alpha \frac{\hat{n} d_{k-j+1}+d_j}{\gcd(d_0,\hat{n}^2-1)}.
 \label{new_digs}
\end{equation}

\noindent Case 2: We now ask when we may construct a new $(k+3)$-digit $(\hat{n},\hat{b})$-palintiple $\hat{p}$ 
with carries $(\hat{c}_{k+2}, \hat{c}_{k+1},\ldots,\hat{c}_0)$ given by $(0,d_{k}, d_{k-1}, \ldots, d_{0},0)$.
For $k+3$ digits we have 
$$
    \hat{d}_j=\frac{\hat{n}\hat{b}\hat{c}_{k-j+3}-\hat{n}\hat{c}_{k-j+2}+\hat{b} \hat{c}_{j+1}-\hat{c}_{j}}{\hat{n}^2-1}
      =\frac{\hat{n}\hat{b}d_{k-j+2}-\hat{n}d_{k-j+1}+\hat{b}d_{j}-d_{j-1}}{\hat{n}^2-1}.
$$
Then $\hat{b}d_0 \equiv 0 \mod(\hat{n}^2-1)$ when $j=0$, so that $\hat{b}=\alpha \frac{\hat{n}^2-1}{\gcd(d_0,\hat{n}^2-1)}$.
It follows that
\begin{equation}
\hat{d}_j=\frac{\alpha}{\gcd(d_0,\hat{n}^2-1)}(\hat{n}d_{k-j+2}+d_{j})-\frac{\hat{n}d_{k-j+1}+d_{j-1}}{\hat{n}^2-1}.
\label{new_digs2}
\end{equation}
 
In order to simplify the exposition, we will say that a palintiple
constructed from a lower-base palintiple is a \textit{derived} palintiple. 
In particular, palintiples derived in the manner described in Case 1 and Case 2 above
will be called \textit{singly-derived} and \textit{doubly-derived} palintiples, respectively.

\begin{remark}
 We note that the above cases may not be the only cases of derived palintiples. 
A computer search for cases of derived palintiples other than singly and doubly-derived
has so far yielded no examples, yet we have not been able to rule out this possibility, and
so we leave it as an open question.
\end{remark}

\section{Palintiples Derived from 1089 Palintiples}

Any $(n,b)$-palintiple for which the Young graph, $Y(n,b)$, is isomorphic to $Y(9,10)$ 
(called a ``1089 graph'' by Sloane \cite{sloane})
shall for the remainder of this article be called a \emph{1089 palintiple}. 
Moreover, the family of palintiples derived from 1089 palintiples
shall, for the purpose of less cumbersome exposition, be called
\textit{Hoey} \footnote{In honor of D. J. Hoey, to whose memory we dedicate this work.}
palintiples.

By Theorem 14 of Kendrick \cite{kendrick_1}, 
we may suppose that $n+1$ divides $b$ with quotient $q$. Furthermore, 
by this same theorem, every node of a 1089 Young graph, $Y(n,b)$, must have the form 
of $[0,0]$, $[n-1,0]$, $[0,n-1]$, or $[n-1,n-1]$,
so that $c_j \equiv 0 \mod(n-1)$ for all $0\leq j \leq k$.
Moreover, by Theorem 6 in \cite{holt}, we may suppose that $c_{k-j}=c_j$ for all $0 \leq j \leq k$
(that is, $p$ is a symmetric palintiple). By the above and Equation \ref{fund}, we have
$d_j=nqr_{k-j+1}+qr_{j+1}-r_j$ for all $0 \leq j \leq k$, where $r_{k-j}=r_j$ equals either 0 or 1
for each $0 \leq j \leq k$.

If $r_1=r_{k-1}\neq 1$, then $d_0=d_k=0$, which violates the assumption of no leading zeros.
If $r_{j-1}=r_{j+1}=1$, but $r_j=0$, then $d_j=b$, which is not a base-$b$ digit.
Similarly,  if $r_{j-1}=r_{j+1}=0$, but $r_j=1$, then $d_j=-1$. 
Thus, $r_k, r_{k-1},\ldots, r_0$ is a palindromic binary sequence such that $r_1=r_{k-1}=1$
and there are no isolated zeros or ones except $r_0=r_k=0$. 

Since $d_0=q$ and $d_k=nq$ by the above computations, we have that
$\gcd(d_0,\hat{n}^2-1)$ divides $d_k$, and that $\hat{b}=n\hat{n}+\alpha \frac{\hat{n}^2-1}{\gcd(nq,\hat{n}^2-1)}$.
Equation \ref{new_digs} then yields
\[
\begin{split}
 \hat{d}_j=&
  nq r_{j-2}+\frac{n^2\hat{n}^2 q -n \hat{n} +\hat{n}-q}{\hat{n}^2-1}r_{j}
 + \frac{n^2 \hat{n} q -n \hat{n}^2 -\hat{n}q+1}{\hat{n}^2-1}r_{j-1}\\
&+\alpha \frac{q r_{j+1} +(\hat{n}nq-1)r_{j}+(nq-\hat{n})r_{j-1}+\hat{n}qr_{j-2}}{\gcd(q,\hat{n}^2-1)}.\\
 \end{split}
\]
To ensure that each term in the above is an integer, we must have that $(n-1)\hat{n} \equiv (n^2-1) q \mod(\hat{n}^2-1)$.
A moment's reflection reveals that $b=q(n+1)$ is the only value for $\hat{n}$ which makes the congruence statement true.
It also ensures that $b$ and $\hat{n}^2-1$ are relatively prime.
Therefore, we have $\hat{b}=nb+\alpha(b^2-1)$ with
\[
 \hat{d}_j =  n^2 q r_{j} + nq r_{j-2} - r_{j-1}+\alpha \left( q r_{j+1}+ (q n b - 1) r_{j}-qr_{j-1} + q b r_{j-2} \right).
\]
Since $d_j<b=\hat{n}$, each $\hat{d}_j$ is less than $\hat{b}$, 
and since there are no singleton ones or zeros in $r_k, r_{k-1},\ldots, r_0$ (except $r_0=r_k)$, 
$\hat{d}_j$ cannot be negative. Thus, each $\hat{d}_j$ is a base-$\hat{b}$ digit.
Additionally, since every $\hat{d}_j$ and $d_j$ satisfy Equation \ref{fund_rel}, 
it follows from a routine calculation that
$(\hat{d}_{k+1},\hat{d}_{k},\ldots,\hat{d}_{0})_{\hat{b}}=b(\hat{d}_{0},\hat{d}_{1},\ldots,\hat{d}_{k+1})_{\hat{b}}$,
where $d_{j-1}$ is the $j$th carry for all $0<j\leq k+1$.
Thus, we have the following:

\begin{theorem}
Suppose $(d_k,d_{k-1}, \ldots, d_0)_b$ is a 1089 $(n,b)$-palintiple.
Then for every $\hat{b}>nb$ such that $\hat{b} \equiv nb \mod (b^2-1)$, 
there exists an asymmetric $(k+2)$-digit $(b,\hat{b})$-palintiple 
with carries $(\hat{c}_{k+1}, \hat{c}_{k},\ldots,\hat{c}_0)$ given by $(d_k,d_{k-1}, \ldots, d_0,0)$.
\label{sd_pals}
\end{theorem}

\begin{example}

The table below contains several examples of Hoey palintiples derived from the $(2,3)$-palintiple $(2,1,2,0,1)_3$,
including the general form obtained from the arguments establishing Theorem \ref{sd_pals}.\\

\begin{tabular}{|c|c|c|}

\hline

$(\hat{n},\hat{b})$ & $(\hat{d}_5, \hat{d}_4, \hat{d}_3, \hat{d}_2, \hat{d}_1, \hat{d}_0)_{\hat{b}}$  & $(\hat{c}_5, \hat{c}_4, \hat{c}_3, \hat{c}_2, \hat{c}_1, \hat{c}_0)$ \\ \hline

$(3,14)$ & $(5, 3, 12, 8, 10, 1)_{14}$  & $(2,1,2,0,1,0)$ \\ \hline

$(3,22)$ & $(8, 5, 19, 13, 16, 2)_{22}$  & $(2,1,2,0,1,0)$ \\ \hline  

$(3,30)$ & $(11, 7, 26, 18, 22, 3)_{30}$  & $(2,1,2,0,1,0)$ \\ \hline

$(3,6+8\alpha)$ & $(2+3\alpha,1+2\alpha,5+7\alpha,3+5\alpha,4+6\alpha,\alpha)_{6+8\alpha}$  & $(2,1,2,0,1,0)$ \\ \hline

\end{tabular}
\label{hoey_example}
\end{example}

\begin{theorem}
 No doubly-derived palintiples can be derived from a 1089 palintiple.
 \label{no_sd2_pals}
\end{theorem}
\begin{proof}
Suppose there exists a doubly-derived $(\hat{n},\hat{b})$-palintiple $\hat{p}$ constructed from a 1089 $(n,b)$-palintiple 
$p=(d_k,d_{k-1}, \ldots, d_0)_b$ with carries $c_k, c_{k-1},\ldots,c_0$. Since $p$ is 1089, we may suppose that 
$n+1$ divides $b$ with quotient $q$, and that $c_j \equiv 0 \mod (n-1)$ as again implied by 
Theorem 14 of Kendrick \cite{kendrick_1}. 
Then $d_0=\frac{bc_1}{n^2-1}=\frac{b(n-1)}{n^2-1}=\frac{b}{n+1}=q$.
Equation \ref{new_digs2} then becomes
$$
\hat{d}_j=\frac{\alpha}{\gcd(q,\hat{n}^2-1)}(\hat{n}d_{k-j+2}+d_{j})-\frac{\hat{n}d_{k-j+1}+d_{j-1}}{\hat{n}^2-1}.
$$
Multiplying both sides by $q(\hat{n}^2-1)$, we have
$$
q(\hat{n}^2-1)\hat{d}_j=(\hat{n}^2-1)\frac{q \alpha}{\gcd(q,\hat{n}^2-1)}(\hat{n}d_{k-j+2}+d_{j})-q(\hat{n}d_{k-j+1}+d_{j-1}).
$$
Reducing modulo $\hat{n}-1$, we have $q(d_{k-j+1}+d_{j-1}) \equiv 0 \mod(\hat{n}-1)$. 
The cases $j=1$ and $j=2$ imply that $q(d_{k}+d_{0}) \equiv 0 \mod(\hat{n}-1)$
and $q(d_{k-1}+d_{1}) \equiv 0 \mod(\hat{n}-1)$. Thus,
$q(nq+q) \equiv 0 \mod(\hat{n}-1)$
and $q(nq-1+q-1) \equiv 0 \mod(\hat{n}-1)$,
which yields $2q \equiv 0 \mod(\hat{n}-1)$ by subtracting the second congruence from the first. Thus, since it is well-known 
that the carries of any palintiple must be less than the multiplier \cite{holt, sloane, sutcliffe, young_1}, 
we have $q=d_0=\hat{c}_1 \leq \hat{n}-1$. Hence, either $q=\frac{\hat{n}-1}{2}$, or $q=\hat{n}-1$. In any case,
$\gcd(q,\hat{n}^2-1)$ divides $\hat{n}-1$, so that $\hat{n}+1$ divides $\hat{b}$ by arguments in Case 2.
But this would imply by Theorem 6 in \cite{holt} that $\hat{p}$ is itself symmetric, so that
$qn=d_{k}=d_0=q$, which is impossible.
\end{proof}

\section{Palintiples Derived From Shifted-Symmetric Palintiples}

We now consider singly-derived palintiples constructed from shifted-symmetric palintiples. 
For brevity, such palintiples will be called \textit{Sutcliffe} palintiples.
We then suppose by Theorem 9 in \cite{holt} that 
$(b-n)c_j \equiv (nb-1)c_j \equiv 0 \mod(n^2-1)$ and $d_j=\frac{(b-n)c_{j+1}+(nb-1)c_j}{n^2-1}$.
Then, by the reasoning of Case 1, we have 
$\hat{b} \frac{(b-n)c_1}{n^2-1} \equiv \hat{n} \frac{(nb-1)c_k}{n^2-1} \mod(\hat{n}^2-1)$.
We shall suppose that $s=\frac{\hat{n}(nb-1)}{b-n}$ is an integer. 
Then, since $c_j=c_{k-j+1}$ for all $0 \leq j \leq k$ by definition, 
$s$ is a particular solution for $\hat{b}$ to the congruence above,
so that in general $\hat{b}=\frac{\hat{n}(nb-1)}{b-n}+\alpha \frac{\hat{n}^2-1}{\gcd(d_0,\hat{n}^2-1)}$.
Therefore, by Equation \ref{new_digs}, 
 \[
 \begin{split}
  \hat{d}_j=&    
   \frac{s(nb-1)-(b-n)}{(\hat{n}-1)(n^2-1)} c_j+\frac{nb-1}{n^2-1}c_{j-1}\\
   &+\frac{\alpha}{\gcd(d_0,\hat{n}^2-1)} \left( \frac{b-n}{n^2-1}c_{j+1}
    +  (\hat{n}+1)\frac{nb-1}{n^2-1}c_j
  +  \hat{n}\frac{b-n}{n^2-1}c_{j-1} \right).\\
  \end{split}
 \]
Supposing that $\hat{n} > d_j$ guarantees, by Equation \ref{fund_rel}, that $0 \leq \hat{d}_j < \hat{b}$. 
Hence, we have the following:
 
\begin{theorem}
Suppose $(d_k,d_{k-1}, \ldots, d_0)_b$ is a shifted-symmetric $(n,b)$-palintiple with carries $c_k, c_{k-1},\ldots,c_0$.
If there exists a natural number $\hat{n}$ such that  $s=\frac{\hat{n}(nb-1)}{b-n}$ is an integer,
and $\hat{n} > d_j$ and $s \frac{(nb-1)c_j}{n^2-1} \equiv \frac{(b-n)c_j}{n^2-1} \mod(\hat{n}-1)$ for all $0 \leq j \leq k$,
then for every $\alpha \geq 1$ such that $\gcd(d_0,\hat{n}^2-1)$ divides 
$\alpha \frac{(b-n)(c_{j+1}+\hat{n}c_{j-1})+(\hat{n}+1)(nb-1)c_j}{n^2-1}$ for all $0 \leq j \leq k$,
an asymmetric $(k+2)$-digit $(\hat{n},\hat{b})$-palintiple exists
with carries $(\hat{c}_{k+1}, \hat{c}_{k},\ldots,\hat{c}_0)$ given by $(d_k,d_{k-1}, \ldots, d_0,0)$,
where $\hat{b}=s+\alpha \frac{\hat{n}^2-1}{\gcd(d_0,\hat{n}^2-1)}$.
\label{ssd1_pals}
\end{theorem}

Theorem \ref{ssd1_pals} will now be applied to a two-digit $(n,b)$-palintiple $p=(d_1,d_0)_b$ 
with one non-zero carry $c$. We note that $p$ is trivially shifted-symmetric. 
Provided that there is an $\hat{n}$ which satisfies the conditions of Theorem \ref{ssd1_pals}, then
 $(\hat{d}_2,\hat{d}_1,\hat{d}_0)_{\hat{b}}$, given by
 \[
 \left(
 \begin{matrix}
 \hat{d}_2\\\\
 \hat{d}_1\\\\
 \hat{d}_0
 \end{matrix}
 \right)
=
\left( \begin{matrix}
  \left(\frac{nb-1}{n^2-1}+\frac{\hat{n} \alpha (b-n)}{\gcd(d_0,\hat{n}^2-1)(n^2-1)}\right) c\\\\
  \left(\frac{s(nb-1)-(b-n)}{(\hat{n}-1)(n^2-1)}+\frac{\alpha(\hat{n}+1)(nb-1)}{\gcd(d_0,\hat{n}^2-1)(n^2-1)}\right)c\\\\
   \frac{\alpha(b-n)}{\gcd(d_0,\hat{n}^2-1)(n^2-1)}c
 \end{matrix}\right)
  =
 \left( \begin{matrix}
  d_1+\alpha \frac{\hat{n}d_0}{\gcd(d_0,\hat{n}^2-1)} \\\\
  \frac{s d_1-d_0}{\hat{n}-1}+\alpha \frac{(\hat{n}+1)d_1}{\gcd(d_0,\hat{n}^2-1)}\\\\
   \alpha \frac{d_0}{\gcd(d_0,\hat{n}^2-1)}
 \end{matrix}\right),
 \]
is a 3-digit $(\hat{n},\hat{b})$-palintiple with carries $(\hat{c}_2,\hat{c}_1,\hat{c}_0)=(d_1,d_0,0)$
for every $\alpha \geq 1$, where $\hat{b}=\frac{\hat{n}(nb-1)}{b-n}+\alpha \frac{\hat{n}^2-1}{\gcd(d_0,\hat{n}^2-1)}$.
Thus, we have the following corollary which provides conditions for the existence of asymmetric palintiples.

\begin{corollary}
If $(d_1,d_0)_b$ is an $(n,b)$-palintiple, and there is an 
$\hat{n} > d_1$ such that $s=\frac{\hat{n}(nb-1)}{b-n}$ is an integer, and
$s d_1 \equiv d_0 \mod(\hat{n}-1)$, 
then asymmetric $(\hat{n},\hat{b})$-palintiples exist,
where $\hat{b}=s+\alpha\frac{\hat{n}^2-1}{\gcd(d_0,\hat{n}^2-1)}$ for any $\alpha \geq 1$.
\label{asym_existence_2}
\end{corollary}

\begin{example}
Consider the $(2,5)$-palintiple $(3,1)_5$ with carries $(c,0)=(1,0)$. We see that 
$\hat{n}=5$ and $\hat{n}=9$ satisfy the conditions of Corollary \ref{asym_existence_2}, from which we get the
$(5,39)$-palintiple $(8,29,1)_{39}$, the $(9,107)$-palintiple $(12,40,1)_{107}$, and
in general, the $(5,5+24\alpha)$-palintiple $(3+5\alpha,11+18\alpha,\alpha)_{15+24\alpha}$,
and the $(9,27+80\alpha)$-palintiple $(3+9\alpha,10+30\alpha,\alpha)_{27+80\alpha}$,
for $\alpha \geq 1$, all with carries $(\hat{c}_2,\hat{c}_1,\hat{c}_0)=(3,1,0)$. 
\end{example}

We now consider doubly-derived palintiples constructed from shifted-symmetric $(n,b)$-palintiples. 
This family of palintiples will be called \textit{Pudwell} palintiples.
Letting $D=\gcd(d_0,\hat{n}^2-1)$, we have by Equation \ref{new_digs2} that 
$D(\hat{n} d_{k-j+1} + d_{j-1}) \equiv 0 \mod(\hat{n}^2-1)$.
Again, by Theorem 9 of \cite{holt}, we replace each digit $d_j$ by $\frac{(b-n)c_{j+1}+(nb-1)c_j}{n^2-1}$,
so that
$$
D\left(\frac{[\hat{n}(nb-1) +(b-n)]c_{j}+ [\hat{n}(b-n)+(nb-1)]c_{j-1}}{n^2-1}\right) \equiv 0 \mod(\hat{n}^2-1).
$$
Since $c_0=0$ by definition, we have by induction over $j$ that
$$
D\frac{[\hat{n} (nb-1) + (b-n)]c_{j}}{n^2-1} \equiv 0 \mod(\hat{n}^2-1)
$$
for all $0 \leq j \leq k$.
Thus, both $\hat{n}-1$ and $\hat{n}+1$ divide $D\frac{[\hat{n} (nb-1) + (b-n)]c_{j}}{n^2-1}$.
Consequently, 
\[
\begin{split}
D\frac{[\hat{n} (nb-1) + (b-n)]c_{j}}{n^2-1}  & \equiv D\frac{[(nb-1) + (b-n)]c_{j}}{n^2-1}\\
 & \equiv D\frac{[(b-1)(n+1)]c_{j}}{n^2-1} \equiv D\frac{(b-1)c_{j}}{n-1} \equiv 0 \mod(\hat{n}-1)\\
\end{split}
\]
and
\[
\begin{split}
D\frac{[\hat{n} (nb-1) + (b-n)]c_{j}}{n^2-1} & \equiv D\frac{[(-1)(nb-1) + (b-n)]c_{j}}{n^2-1}\\ 
& \equiv D\frac{[-((b+1)(n-1))]c_{j}}{n^2-1} \equiv D\frac{-(b+1)c_{j}}{n+1} \equiv 0 \mod(\hat{n}+1).\\
\end{split}
\]
That is, $\hat{n}-1$ and $\hat{n}+1$ must divide $D\frac{(b-1)c_j}{n-1}$ and $D\frac{(b+1)c_j}{n+1}$, respectively.

Using the above conclusion and a computer, we have found no examples for which $\hat{n} \neq b$.
However, we have not been able to rule out this possibility. 
Checking all possibilities for all $b \leq 500$ yielded no Pudwell palintiples for which $\hat{n}$ and $b$ are not equal.
Therefore, we shall narrow our scope and consider the case $\hat{n}=b$ while leaving the $\hat{n} \neq b$ case
as an open problem.

By the above arguments, we may say for each $c_j$ that
$$
D\frac{[\hat{n}(nb-1) + (b-n)]c_{j}}{n^2-1}=Q_j(\hat{n}^2-1)
$$
for some integer $Q_j$. Replacing $\hat{n}$ with $b$, we have
$$D\left(\frac{[b(nb-1) + (b-n)]c_{j}}{n^2-1}\right)=Q_j(b^2-1).$$
But $D\left(\frac{[b(nb-1) + (b-n)]c_{j}}{n^2-1}\right)=\frac{(b^2-1)nDc_{j}}{n^2-1}$,
so that $$\frac{n Dc_{j}}{n^2-1}=Q_j.$$
Since $n$ and $n^2-1$ are relatively prime, 
we have that $n^2-1$ divides each $D c_j$ with some quotient $q_j$ for all $0 \leq j \leq k$.

Once again, replacing each digit $d_j$ in Equation \ref{new_digs2} with $\frac{(b-n)c_{j+1}+(nb-1)c_j}{n^2-1}$, 
and applying the above result that $D c_j=(n^2-1)q_j$ for all $0 \leq j \leq k$, we obtain
$$
\hat{d}_j=\frac{\alpha(b d_{k-j+2}+d_{j})-(n q_j+q_{j-1})}{D}.
$$
As argued previously, $0 \leq \hat{d}_j<\hat{b}$ since
each $d_j<b=\hat{n}$. We therefore have the following:

\begin{theorem}
Suppose $(d_k,d_{k-1}, \ldots, d_0)_b$ is a shifted-symmetric $(n,b)$-palintiple with carries $c_k, c_{k-1},\ldots,c_0$
and let $D=\gcd(d_0,b^2-1)$. If $n^2-1$ divides $D c_j$ with quotient $q_j$ for all $0 \leq j \leq k$, 
then for every $\alpha \geq 1$ such that $D$ divides $\alpha (bd_{k-j+2}+d_{j})-(n q_j+q_{j-1})$ for all $0\leq j \leq k$, 
a $(k+3)$-digit asymmetric $(b,\hat{b})$-palintiple exists with carries 
$(\hat{c}_{k+2}, \hat{c}_{k+1},\ldots,\hat{c}_0)$ given by $(0,d_{k}, d_{k-1}, \ldots, d_{0},0)$, where 
$\hat{b}=\alpha \frac{b^2-1}{D}$.
\label{ssd2_pals}
\end{theorem}

The case $k=1$ gives us another condition which guarantees the existence of asymmetric palintiples.

\begin{corollary}
Suppose $(d_1,d_0)_b$ is an $(n,b)$-palintiple with one non-zero carry $c$ and $D=\gcd(d_0,b^2-1)$.
If $n^2-1$ divides $Dc$ with quotient $q$, and $\gcd(d_1,D)$ divides $q$,
then there exists an asymmetric $(b,\hat{b})$-palintiple, where $\hat{b}=\alpha \frac{b^2-1}{D}$.
\label{asym_existence_4}
\end{corollary}

\begin{proof}
The arguments leading up to Theorem \ref{ssd2_pals} give us that a new 4-digit palintiple
  $
  (\hat{d}_3, \hat{d}_2,\hat{d}_1, \hat{d}_0)_{\hat{b}}
  $
with carries $(0,d_1,d_0,0)$ must equal 
  $
  (\frac{\alpha b d_0}{D}, \frac{\alpha bd_1-q}{D},\frac{\alpha d_1-nq}{D},\frac{\alpha d_0}{D})_{\alpha \frac{b^2-1}{D}}
  $
for some $\alpha$, where $k=1$. Since $\gcd(d_1,D)$ divides $q$, there is an $\alpha \geq 1$ such that 
$\alpha b d_1 \equiv q \mod D$, and since $(nb-1)q=d_1 D$, we have $\alpha d_1 \equiv nq \mod D$.
 \end{proof}

 \begin{example}
 A family of Pudwell palintiples may be constructed from the $(6,55)$-palintiple $(47,7)_{55}$ with carries $(c,0)=(5,0)$.
 The conditions of Corollary \ref{asym_existence_4} are satisfied, and we have that
$(\hat{d}_3, \hat{d}_2,\hat{d}_1, \hat{d}_0)_{\hat{b}}$ given by 
$\left(55 \alpha,\frac{ 2585 \alpha-1}{7}, \frac{47 \alpha-6}{7},\alpha\right)_{432\alpha}$ is a
$(55,432\alpha)$-palintiple with carries $(\hat{c}_3,\hat{c}_2,\hat{c}_1,\hat{c}_0)=(0,47,7,0)$, 
where $\alpha$ is any natural number congruent to 4 modulo 7.
 \end{example}

\section{Palintiples Derived from Palintiple Reversals}

Digit reversals of palintiples also appear in the carries of higher-base palintiples. 
Therefore, we now construct asymmetric palintiples from digit-reversals of palintiples.
We will not present the amount of detail as in the previous sections, as the arguments are essentially the same
for each case. However, we will highlight points which deserve additional explanation.  
We shall consider both $(k+2)$-digit $(\hat{n},\hat{b})$-palintiples with carries
$(\hat{c}_{k+1}, \hat{c}_{k},\ldots,\hat{c}_0)$ of the form $(d_0,d_{1}, \ldots, d_{k},0)$, 
and $(k+3)$-digit $(\hat{n},\hat{b})$-palintiples with carries
$(\hat{c}_{k+2}, \hat{c}_{k+1},\ldots,\hat{c}_0)$ of the form $(0,d_0,d_{1}, \ldots, d_{k},0)$.
Such palintiples will be called \textit{singly-$\rho$-derived} and \textit{doubly-$\rho$-derived}, respectively.

\subsection{Palintiples Derived from Reversals of 1089 Palintiples}

We shall now consider families of singly-$\rho$-derived palintiples constructed from 1089 palintiples.
These shall be called \textit{$\rho$-Hoey} palintiples.
In a manner similar to the arguments leading to Equation \ref{new_digs}, 
it must be that $\hat{b}d_k \equiv \hat{n} d_0 \mod(\hat{n}^2-1)$, or
$\hat{b} n q \equiv \hat{n} q \mod(\hat{n}^2-1)$. Thus, in order for a solution $\hat{b}$ to exist, 
we require both that $\gcd(nq,\hat{n}^2-1)$ divide $\hat{n}q$ 
(and consequently, $q$, since $\hat{n}$ and $\hat{n}^2-1$ are relatively prime), 
and that $n$ and $\hat{n}^2-1$ are relatively prime. Under these assumptions, we then have that 
$\hat{b}=m \hat{n} + \alpha \frac{\hat{n}^2-1}{\gcd(nq,\hat{n}^2-1)}$, where $m$ is the multiplicative inverse 
of $n$ modulo $\hat{n}^2-1$. 
Reparameterizing, we let $\ell$ be the least non-negative residue of $m\hat{n}$ 
modulo $\frac{\hat{n}^2-1}{\gcd(nq,\hat{n}^2-1)}$, so that 
$\hat{b}=\ell + \alpha \frac{\hat{n}^2-1}{\gcd(nq,\hat{n}^2-1)}$ for $\alpha \geq 1$.
Then
\[
\begin{split}
\hat{d}_j
=&\frac
{(\ell n-\hat{n}) qr_{j+1}+(\hat{n}-\ell-nq + \ell \hat{n} q)r_{j}+(1-\hat{n} \ell -\hat{n} n q+\ell q)r_{j-1}+(\ell n \hat{n}-1)qr_{j-2}}
{\hat{n}^2-1}\\
&+\alpha \frac{nqr_{j+1}+(\hat{n}q-1)r_{j}+(q-\hat{n})r_{j-1}+\hat{n}nq r_{j-2}}{\gcd(nq,\hat{n}^2-1)}.
\end{split}
\]
Since $\ell n q \equiv \hat{b} n q \equiv \hat{n} q \mod(\hat{n}^2-1)$, we have both 
$(\ell n-\hat{n})q \equiv 0 \mod(\hat{n}^2-1)$ and $(\ell n \hat{n}-1)q \equiv 0 \mod(\hat{n}^2-1)$.
Thus, in order to ensure that the above is an integer, 
we shall require  $\hat{n}-\ell-nq + \ell \hat{n} q \equiv 0 \mod(\hat{n}^2-1)$ and 
$1-\hat{n} \ell -\hat{n} n q+\ell q \equiv 0 \mod(\hat{n}^2-1)$, 
both of which are equivalent as seen by multiplying the first congruence by $\hat{n}$.
Since $\gcd(nq,\hat{n}^2-1)$ divides $q$, as mentioned above, we have that 
$m \hat{n} q \equiv \hat{b}q \equiv \ell q \mod(\hat{n}^2-1)$.
Multiplying the second congruence above by $q$ and then substituting $\ell q$ with $m \hat{n} q$, we obtain
$q- m q -\hat{n} n q^2+m \hat{n} q \equiv 0 \mod(\hat{n}^2-1)$. Multiplying by $n$, the multiplicative inverse of $m$,
we then have $(n-1)\hat{n}q \equiv (n^2-1) q^2 \mod(\hat{n}^2-1)$. Therefore, as before,  we have that $\hat{n}=b=q(n+1)$, 
so that $\gcd(nq,\hat{n}^2-1)=\gcd(nq,b^2-1)=1$. Thus,
\[
\begin{split}
\hat{d}_j
=&\frac
{(\ell n-b) qr_{j+1}+(b-\ell-nq + \ell b q)r_{j}+(1-b \ell -b n q+\ell q)r_{j-1}+(\ell n b-1)qr_{j-2}}
{b^2-1}\\
&+\alpha (nqr_{j+1}+(bq-1)r_{j}+(q-b)r_{j-1}+bnq r_{j-2}).
\end{split}
\]

The above arguments give us the $\rho$-derived compliment to Theorem \ref{sd_pals}.

\begin{theorem}
Suppose $(d_k,d_{k-1}, \ldots, d_0)_b$ is a 1089 $(n,b)$-palintiple such that
$b^2-1$ and $n$ are relatively prime, and $m$ is the multiplicative inverse of $n$ modulo $b^2-1$.
Furthermore, let $\ell$ be the least non-negative residue of $mb$ modulo $b^2-1$.
Then for every $\hat{b} > \ell$ such that $\hat{b} \equiv \ell \mod (b^2-1)$, 
there exists an asymmetric $(k+2)$-digit $(b,\hat{b})$-palintiple 
with carries $(\hat{c}_{k+1}, \hat{c}_{k},\ldots,\hat{c}_0)$ given by $(d_0,d_{1}, \ldots, d_{k}, 0)$.
\label{sd_pals_rev}
\end{theorem}

\begin{example}
 Applying the arguments for Theorem \ref{sd_pals_rev} to the well-known
 $(4,10)$-palintiple  $(8, 7, 1, 2)_{10}$  with carries $(c_3,c_2,c_1,c_0)=(0, 3, 3, 0)$,
 we have $\ell=52$, which gives rise to the  family of $(10,52+99\alpha)$-palintiples 
 with digits $(\hat{d}_4, \hat{d}_3, \hat{d}_2, \hat{d}_1, \hat{d}_0)_{\hat{b}}$ given by
 $(42+80\alpha, 37+72\alpha, 5+11\alpha, 14+27\alpha, 4+8\alpha)_{52+99\alpha}$,
 all with carries $(\hat{c}_4, \hat{c}_3, \hat{c}_2, \hat{c}_1,\hat{c}_0)=(2, 1, 7, 8, 0)$ for all $\alpha \geq 1$.
\end{example}

\begin{remark}
  Although it is not always the case, the example above also yields a palintiple for $\alpha=0$.
\end{remark}

It can be shown by an argument nearly identical to that of Theorem \ref{no_sd2_pals} that 
no doubly-$\rho$-derived palintiples can be constructed from reversals of 1089 palintiples.

\subsection{Palintiples Derived from Reversals of Shifted-Symmetric Palintiples}

Singly-$\rho$-derived palintiples constructed from shifted-symmetric palintiples 
(called $\rho$-Sutcliffe palintiples) yield an argument and theorem statement nearly identical to 
that of Theorem \ref{ssd1_pals}, with the exception that the roles of $b-n$ and $nb-1$, as well as 
$d_0$ and $d_k$, are interchanged. 
 
\begin{theorem}
Suppose $(d_k,d_{k-1}, \ldots, d_0)_b$ is a shifted-symmetric $(n,b)$-palintiple with carries $c_k, c_{k-1},\ldots,c_0$.
If there exists a natural number $\hat{n}$ such that  $s=\frac{\hat{n}(b-n)}{nb-1}$ is an integer,
and $\hat{n} > d_j$ and $s \frac{(b-n)c_j}{n^2-1} \equiv \frac{(nb-1)c_j}{n^2-1} \mod(\hat{n}-1)$ for all $0 \leq j \leq k$,
then for every $\alpha \geq 1$ such that $\gcd(d_k,\hat{n}^2-1)$ divides 
$\alpha \frac{(nb-1)(c_{j+1}+\hat{n}c_{j-1})+(\hat{n}+1)(b-n)c_j}{n^2-1}$ for all $0 \leq j \leq k$,
an asymmetric $(k+2)$-digit $(\hat{n},\hat{b})$-palintiple exists
with carries $(\hat{c}_{k+1}, \hat{c}_{k},\ldots,\hat{c}_0)$ given by $(d_0,d_{1}, \ldots, d_k,0)$,
where $\hat{b}=s+\alpha \frac{\hat{n}^2-1}{\gcd(d_k,\hat{n}^2-1)}$.
\label{ssd1_pals_rev}
\end{theorem}

\begin{corollary}
If $(d_1,d_0)_b$ is an $(n,b)$-palintiple, and there exists an 
$\hat{n} > d_1$ such that $s=\frac{\hat{n}(b-n)}{nb-1}$ is an integer, and
$s d_0 \equiv d_1 \mod(\hat{n}-1)$, 
then asymmetric $(\hat{n},\hat{b})$-palintiples exist,
where $\hat{b}=s+\alpha\frac{\hat{n}^2-1}{\gcd(d_1,\hat{n}^2-1)}$ for any $\alpha \geq 1$.
\label{asym_existence_5}
\end{corollary}

\begin{example}
 Corollary \ref{asym_existence_5} applies to the $(2,5)$-palintiple $(3,1)_5$  with one nontrivial carry $c=1$. 
 The value $\hat{n}=9$ satisfies its hypotheses, giving us the family of $(9,3+80\alpha)$-palintiples
 $(1+27\alpha,10\alpha,3 \alpha)_{3+80\alpha}$, where $\alpha$ is any natural number, each with carries 
 $(\hat{c}_2,\hat{c}_1,\hat{c}_0)=(1,3,0)$.
 \end{example}

Considering doubly-$\rho$-derived palintiples constructed from shifted-symmetric palintiples ($\rho$-Pudwell),
we obtain a $\rho$-derived compliment to Theorem \ref{ssd2_pals} whose argument (like that of Theorem \ref{ssd1_pals_rev})
transposes $b-n$ and $nb-1$, as well as $d_0$ and $d_k$. 

\begin{theorem}
Suppose $(d_k,d_{k-1}, \ldots, d_0)_b$ is a shifted-symmetric $(n,b)$-palintiple with carries $c_k, c_{k-1},\ldots,c_0$,
and let $D=\gcd(d_k,b^2-1)$. If $n^2-1$ divides $D c_j$ with quotient $q_j$ for all $0 \leq j \leq k$, 
then for every $\alpha \geq 1$ such that $D$ divides $\alpha (bd_{j-2}+d_{k-j})-(q_j+nq_{j-1})$ for all $0\leq j \leq k$, 
a $(k+3)$-digit asymmetric $(b,\hat{b})$-palintiple exists with carries 
$(\hat{c}_{k+2}, \hat{c}_{k+1},\ldots,\hat{c}_0)$ given by $(0,d_{0}, d_{1}, \ldots, d_{k},0)$, where 
$\hat{b}=\alpha \frac{b^2-1}{D}$.
\label{ssd2_pals_rev}
\end{theorem}

\begin{corollary}
Suppose $(d_1,d_0)_b$ is an $(n,b)$-palintiple with one non-zero carry $c$, and $D=\gcd(d_1,b^2-1)$.
If $n^2-1$ divides $Dc$ with quotient $q$, and $\gcd(d_0,D)$ divides $q$, 
then there exists an asymmetric $(b,\hat{b})$-palintiple, where $\hat{b}=\alpha \frac{b^2-1}{D}$.
\label{asym_existence_6}
\end{corollary}

\begin{example}
We again look at the $(2,5)$-palintiple $(3,1)_5$ with carry $c=1$. 
The conditions of Corollary \ref{asym_existence_6} are satisfied, and we get the
$(5, 8 \alpha)$-palintiple $(5\alpha,\frac{5 \alpha-2}{3},\frac{\alpha-1}{3},\alpha)_{8\alpha}$
with carries $(\hat{c}_3,\hat{c}_2,\hat{c}_1,\hat{c}_0)=(0,1,3,0)$, where 
$\alpha \equiv 1 \mod 3$. 
\end{example}

\section{Palinomials and Derived Palintiples}

We recall a definition from \cite{holt}:
the $(n,b)$-\textit{palinomial} induced by an $(n,b)$-palintiple $(d_k, \ldots,d_0)_b$ is the polynomial
$$\mbox{Pal}(x)=\sum_{j=0}^{k}(d_{j}-nd_{k-j})x^j.$$

\begin{theorem}
Palinomials induced by 1089 palintiples have at least one root on the unit circle.
\end{theorem}

\begin{proof}
By Theorem 11 in \cite{holt}, $\mbox{Pal}(x)=(x-b)\sum_{j=1}^{k}c_jx^{j-1}$,
and since this palinomial is induced by a 1089 palintiple, 
we have by Theorem 14 in \cite{kendrick_1} that $n+1$ divides $b$.
Hence, by Theorem 6 in \cite{holt} and Theorem 14 in \cite{kendrick_1}, $c_j=c_{k-j}$ is either 0 or $n-1$. 
Our palinomial then has the form $\mbox{Pal}(x)=(n-1)(x-b)\sum_{j=1}^{k}r_jx^{j-1}$,
where $r_k, r_{k-1},\ldots, r_0$ is a palindromic binary sequence such that $r_1=r_{k-1}=1$,
and there are no isolated zeros or ones except $r_0=r_k=0$ as already argued.
Corollary 1 of \cite{poly} proves that any palindrome polynomial with coefficients which are either
0 or 1 always has a unimodular root, and this result establishes our claim.  
\end{proof}

The next theorem reveals an even closer connection between the digits of 1089 and shifted-symmetric palintiples
and the roots of their palinomials.

\begin{theorem}
 Let $\xi \neq b$ be a non-zero root of the palinomial induced by a 1089 or shifted-symmetric palintiple 
 $(d_k, d_{k-1},\ldots, d_0)_b$.
 Then $\xi$ is a root of both the digit and reverse-digit polynomials. That is,
 $$\sum_{j=0}^n d_j \xi^j=\sum_{j=0}^n d_{k-j} \xi^j=0.$$
 \label{dig_poly}
\end{theorem}

\begin{proof}
 By the theorem hypothesis, $\sum_{j=1}^{k}c_jx^{j-1}$ is a palindromic polynomial.
 Therefore, $\mbox{Pal}(\xi)=\mbox{Pal}(\frac{1}{\xi})=0$.
 It follows that both $\sum_{j=0}^{k}d_j \xi^j=n\sum_{j=0}^{k}d_{k-j}\xi^j$ and 
 $\sum_{j=0}^{k}d_j \xi^{-j}=n\sum_{j=0}^{k}d_{k-j}\xi^{-j}$. 
 Multiplying the second equation by $n \xi^k$ and reindexing the sum, we have
 $n\sum_{j=0}^{k}d_{k-j} \xi^j=n^2\sum_{j=0}^{k}d_{j}\xi^j$. 
 Hence, $\sum_{j=0}^{k}d_j \xi^j=n^2\sum_{j=0}^{k}d_{j}\xi^j$,
 or $(n^2-1)\sum_{j=0}^{k}d_{j}\xi^j=0$,
 so that $\xi$ is a zero of the forward-digit polynomial.
 The reverse-digit case then follows from the above relation:
 $n\sum_{j=0}^{k}d_{k-j}\xi^j=\sum_{j=0}^{k}d_j \xi^j=0$.
 
\end{proof}

\begin{corollary} 
Digit and reverse-digit polynomials of 1089 palintiples have at least one root on the unit circle. 
\end{corollary}

\subsection{Additional Roots of Digit Polynomials of 1089 and Shifted-Symmetric Palintiples}

By Theorem \ref{dig_poly}, every negative or purely complex root of 
the $(k-1)$st-degree palinomial induced by a $(k+1)$-digit 1089 palintiple
is also a root of the digit polynomial (the degree is $k-1$ since $c_k=c_0=0$ by palintiple symmetry).
Thus, both the digit and reverse-digit polynomial of a 
1089 palintiple have two additional roots more than their corresponding palinomial. 

\begin{theorem}
Let $\mbox{Pal}(x)$ be the palinomial induced by 
a 1089 $(n,b)$-palintiple $(d_k, d_{k-1},\ldots, d_0)_b$, and let
$D$ and $\overline{D}$ denote the digit and reverse-digit polynomials, respectively. 
Then $$D(x)=(d_k x^2-x+d_0)\frac{\mbox{Pal}(x)}{(n-1)(x-b)} 
\mbox{ and } \overline{D}(x)=(d_0 x^2-x+d_k)\frac{\mbox{Pal}(x)}{(n-1)(x-b)}.$$
\end{theorem}

\begin{proof}
Suppose $\mbox{Pal}(x)=(n-1)(x-b)\prod_{j=1}^{k-2}(x-\xi_j)$ is a palinomial induced by a 1089 (symmetric) $(n,b)$-palintiple.
By Theorem \ref{dig_poly}, we may express these as $D(x)=d_k(x-\omega_1)(x-\omega_2) \prod_{j=1}^{k-2}(x-\xi_j)$
and $\overline{D}(x)=d_0(x-\frac{1}{\omega_1})(x-\frac{1}{\omega_2}) \prod_{j=1}^{k-2}(x-\xi_j)$,
where $\omega_1$ and $\omega_2$ are the two extra roots. By Corollary 12 in \cite{holt}, the only 
positive real root of a palinomial is $b$. Thus, $x=1$ 
cannot be a root of any palinomial and is clearly not a root of $D$ or $\overline{D}$.  
Thus, since $D(1)=\overline{D}(1)$, and since $d_0=q$ and $d_k=nq$ for any 1089 palintiple, we have that
$n(1-\omega_1)(1-\omega_2)=(1-\frac{1}{\omega_1})(1-\frac{1}{\omega_2})$. Then
$n(1-\omega_1)(1-\omega_2)=\frac{\omega_1-1}{\omega_1}\frac{\omega_2-1}{\omega_2}$, so that
after cancelling common factors we have $\omega_1 \omega_2=\frac{1}{n}$.
Now, $D(b)=n\overline{D}(b)$ implies $d_k(b-\omega_1)(b-\omega_2)=n d_0(b-\frac{1}{\omega_1})(b-\frac{1}{\omega_2})$,
so that by the same reasoning as above, $(b-\omega_1)(b-\omega_2)=(b-\frac{1}{\omega_1})(b-\frac{1}{\omega_2})$.
Expanding both sides, we then have 
$\omega_1\omega_2-b\omega_1-b\omega_2=\frac{1}{\omega_1 \omega_2}-\frac{b}{\omega_1}-\frac{b}{\omega_2}$, which
after rearranging becomes
$\omega_1 \omega_2-b(\omega_1+\omega_2)=\frac{1}{\omega_1 \omega_2}-b\frac{\omega_1+\omega_2}{\omega_1\omega_2}$.
Thus, since $\omega_1\omega_2=\frac{1}{n}$, as demonstrated above, the previous statement becomes
$\frac{1}{n}-b(\omega_1+\omega_2)=n-nb(\omega_1+\omega_2)$. Solving for $\omega_1+\omega_2$, we obtain
$\omega_1+\omega_2=\frac{n-\frac{1}{n}}{nb-b}$,
which, since $b=q(n+1)$, simplifies to $\omega_1+\omega_2=\frac{1}{nq}$. 
We again cite the fact that the only positive real root of a palinomial is $b$, and since $\omega_1$ and $\omega_2$
add to a positive number, we conclude that these must be complex. 
Furthermore, since all other conjugate pairs of $\mbox{Pal}(x)$ were cancelled in the above calculations,
the only conclusion is that $\omega_1$ and $\omega_2$ are conjugate.
Suppose then that $\omega_1=x+iy$ and $\omega_2=x-iy$. Then $2x=\omega_1+\omega_2=\frac{1}{n}$, so that
the real part of both roots is $\frac{1}{2nq}$. Now, since $\omega_1\omega_2=\frac{1}{n}$,
we have $(\frac{1}{2qn}+iy)(\frac{1}{2qn}-iy)=\frac{1}{n}$, which implies $y=\pm\sqrt{\frac{1}{n}-\frac{1}{4n^2q^2}}$.
It is then a straightforward calculation to determine that $\omega_1$ and $\omega_2$ are
the conjugate pair $\frac{1}{2nq}(1 \pm i\sqrt{4q^2n-1})$.
The digit and reverse-digit polynomials may then be expressed as 
$D(x)=(d_k x^2-x+d_0)\prod_{j=1}^{k-2}(x-\xi_j)$ and
$\overline{D}(x)=(d_0 x^2-x+d_k)\prod_{j=1}^{k-2}(x-\xi_j)$.
\end{proof}

Another application of Theorem \ref{dig_poly} shows that
every negative or purely complex root of the $k$th-degree palinomial induced by a $(k+1)$-digit shifted-symmetric palintiple
is also a root of the digit polynomial (the degree is $k$ since $c_k=c_1 \neq 0$ by shifted-symmetry).
Thus, both the digit and reverse-digit polynomial of a 
shifted-symmetric palintiple have one more root than their corresponding palinomial.  

\begin{theorem}
Let $\mbox{Pal}(x)$ be the palinomial induced by a shifted-symmetric $(n,b)$-palintiple 
$(d_k, d_{k-1},\ldots, d_0)_b$ with carries $c_k, c_{k-1}, \ldots, c_1, c_0$,  and let
$D$ and $\overline{D}$ denote the digit and reverse-digit polynomials, respectively. 
Then $$D(x)=(d_k x + d_0)\frac{\mbox{Pal}(x)}{c_k (x-b)} 
\mbox{ and } \overline{D}(x)=(d_0 x + d_k)\frac{\mbox{Pal}(x)}{c_k(x-b)}.$$
\end{theorem}

\begin{proof}
Suppose $\mbox{Pal}(x)=c_k(x-b)\prod_{j=1}^{k-1}(x-\xi_j)$.
Since the digit and reverse-digit polynomials have one more root $\omega$ than $\mbox{Pal}(x)$, we have that
$D(x)=d_k(x-\omega) \prod_{j=1}^{k-1}(x-\xi_j)$ and $\overline{D}(x)=d_0(x-\frac{1}{\omega})\prod_{j=1}^{k-1}(x-\xi_j)$.
Thus, since $x=1$ cannot be the root of any palinomial as argued in the proof of the previous theorem, 
it follows from the fact that $D(1)=\overline{D}(1)$ that $d_k(1-\omega)=d_0(1-\frac{1}{\omega})$ 
after cancelling common factors. Then, multiplying by $\omega$, we have $d_k \omega (1-\omega)=d_0(\omega-1)$, which implies
$\omega=-\frac{d_0}{d_k}$.
By Theorem 9 of \cite{holt}, we have that $d_k=\frac{(nb-1)c_1}{n^2-1}$ and $d_0=\frac{(b-n)c_1}{n^2-1}$,
so that $\omega=-\frac{b-n}{nb-1}$.
Hence, $D(x)=(d_k x+d_0) \prod_{j=1}^{k-1}(x-\xi_j)$ and $\overline{D}(x)=(d_0 x+d_k) \prod_{j=1}^{k-1}(x-\xi_j)$.
\end{proof}

\begin{corollary}
Let $\widehat{\mbox{Pal}}(x)$ be the palinomial induced by a singly-derived or doubly-derived 
$(\hat{n},\hat{b})$-palintiple $\hat{p}$ constructed from an $(n,b)$-palintiple $p=(d_k, d_{k-1},\ldots, d_0)_b$, 
and let $\mbox{Pal}(x)$ be the palinomial induced by $p$. Then
$$\widehat{\mbox{Pal}}(x)=(x-\hat{b})(d_k x^2-x+d_0)\frac{\mbox{Pal}(x)}{(n-1)(x-b)} $$
if $p$ is a 1089 palintiple, and
$$\widehat{\mbox{Pal}}(x)=(x-\hat{b})(d_k x+d_0)\frac{\mbox{Pal}(x)}{c_k(x-b)} $$
if $p$ is shifted-symmetric, where $c_k$ is the $k$th carry of $p$.
\end{corollary}

\begin{proof}
If $\hat{p}$ is singly-derived, its carries are $d_k$, $d_{k-1}, \ldots, d_0, 0$, 
so that by Theorem 11 in \cite{holt}, we have
$
\widehat{\mbox{Pal}}(x)
=(x-\hat{b})\sum_{j=1}^{k+1}\hat{c}_j x^{j-1}
=(x-\hat{b})\sum_{j=1}^{k+1}d_{j-1} x^{j-1}
=(x-\hat{b})D(x).
$ 
The doubly-derived case follows in a similar fashion. 
\end{proof}

\begin{corollary}
Let $\widehat{\mbox{Pal}}(x)$ be the palinomial induced by a singly-$\rho$-derived or doubly-$\rho$-derived 
$(\hat{n},\hat{b})$-palintiple $\hat{p}$
constructed from an $(n,b)$-palintiple $p=(d_k, d_{k-1},\ldots, d_0)_b$,
and let $\mbox{Pal}(x)$ be the palinomial induced by $p$. Then
$$\widehat{\mbox{Pal}}(x)=(x-\hat{b})(d_0 x^2-x+d_k)\frac{\mbox{Pal}(x)}{(n-1)(x-b)} $$
if $p$ is a 1089 palintiple, and
$$\widehat{\mbox{Pal}}(x)=(x-\hat{b})(d_0 x+d_k)\frac{\mbox{Pal}(x)}{c_k(x-b)} $$
if $p$ is shifted-symmetric, where $c_k$ is the $k$th carry of $p$.
\end{corollary}

\begin{corollary} 
Palinomials induced by Hoey and $\rho$-Hoey palintiples have at least one root on the unit circle. 
\end{corollary}

\begin{corollary}
Palinomials induced by any two Hoey palintiples derived from a common palintiple differ only by a linear factor. 
\label{linear_factor}
\end{corollary}

The statement of Corollary \ref{linear_factor} also holds for $\rho$-Hoey, Sutcliffe, $\rho$-Sutcliffe, 
Pudwell, and $\rho$-Pudwell palintiples.

\begin{example}
The 7-digit 1089 $(4,10)$-palintiple $p=(8,7,9,9,9,1,2)_{10}$ 
induces the palinomial $\mbox{Pal}(x)=3(x-10)(x^4+x^3+x^2+x+1)$.
The reader may also verify that $D(x)=(8 x^2-x+2)(x^4+x^3+x^2+x+1)$ 
and $\overline{D}(x)=(2x^2-x+8)(x^4+x^3+x^2+x+1)$.
Moreover, constructing a new 8-digit palintiple from $p$ using Theorem \ref{sd_pals} and its supporting arguments, 
we take the $(10,139)$-palintiple $\hat{p}=(28,25,136,138,138,110,113,2)_{139}$ as an example.
The reader may verify that the palinomial induced by $\hat{p}$ can be expressed as
$$
\widehat{\mbox{Pal}}(x)=(x-139)(8 x^2-x+2)(x^4+x^3+x^2+x+1).
$$
\end{example}

\section{Open Questions and Future Work}

It is still unknown if Theorems \ref{sd_pals} and \ref{sd_pals_rev} and their arguments 
give us all Hoey and $\rho$-Hoey palintiples, respectively.
So far this seems to be the case, but remains unproven. 
On the  other hand, however, Sutcliffe and $\rho$-Sutcliffe palintiples exist under conditions for which 
Theorems \ref{ssd1_pals}, \ref{ssd1_pals_rev}, and their corollaries do not apply. 
We give as an example the $(14,129)$-palintiple $(37, 89, 2)_{129}$ with carries $(9,4,0)$, 
which is derived from the $(2,14)$-palintiple $(9,4)_{14}$; for this particular case, $s=\frac{\hat{n}(nb-1)}{b-n}$ is not an integer.
Furthermore, we have already stated that it is unknown if there are Pudwell palintiples such that $\hat{n} \neq b$. 
However, we must point out that $\rho$-Pudwell palintiples do exist for values of $\hat{n}$ other than $b$. 
As an example, we present the $(34,55)$-palintiple $(34, 1, 0, 1)_{55}$ with carries $(0,1,21,0)$, which is derived
from the $(11,23)$-palintiple $(21,1)_{23}$. Moreover, for whatever reason, $\rho$-Pudwell palintiples,
at least for lower bases, seem to occur much more frequently than their forward counterparts.
Both Pudwell and $\rho$-Pudwell palintiples so far have proven to be the least well-understood. 
In summary, finding maximal conditions for the existence of palintiples belonging to the 
families presented in this paper is an open topic.

Given the variety of Young graph isomorphism classes \cite{kendrick_1,kendrick_2},
it is not surprising that not all asymmetric palintiples are derived palintiples. 
If we consider the example of the $(4,23)$-palintiple $(6,15,1)_{23}$ with carries $(c_2,c_1,c_0)=(2,1,0)$,
it is not difficult to show that no 2-digit palintiple has these carries as digits.
We might then ask if there is a more general principle at work here; 
perhaps the carries are not the digits of a palintiple, but rather, the digits of a \textit{permutiple}. 
Indeed, one sees that for the above case
that $(2, 1, 0)_4 = 2 \cdot (1, 0, 2)_4$. 
While such examples are promising, one can verify that for the $(11,17)$-palintiple $(14, 12, 5, 1)_{17}$, 
there is no permutation, base, or multiplier for which the carries 
$(c_3,c_2,c_1,c_0)=(3, 8, 9, 0)$ are a non-trivial permutiple. 
However, we do point out that there do seem to be strong connections, and naturally so, 
between palintiples and the more general permutiple problem. 
Thus, a more developed understanding of permutiples may very well provide a better understanding of palintiples. 

With the above in mind, we also mention that it is unknown 
if singly-derived or doubly-derived palintiples can be constructed 
from other asymmetric palintiples (neither 1089 nor shifted-symmetric).  
So far none have been found.
Moreover, as mentioned at the end of Section 2, 
it is still unknown if cases of ``triply,'' ``quadruply,'' or similarly derived palintiples exist.
A single case has yet to be found.

Another important unanswered question mentioned by \cite{holt} involves symmetric palintiples. 
It is conjectured in \cite{holt} that a palintiple is symmetric if and only if $n+1$ divides $b$.
Kendrick \cite{kendrick_1} showed that $Y(n,b)$ is isomorphic to $Y(9,10)$ if and only if $n+1$ divides $b$.
Thus, we ask if the following are equivalent for an $(n,b)$-palintiple $p=(d_k, d_{k-1},\ldots, d_0)_b$
with carries $c_k, c_{k-1}, \ldots, c_1, c_0$: 
\begin{enumerate}
\item $p$ is symmetric,
\item $p$ is 1089,
\item $c_j \equiv 0 \mod(n-1)$ for all $0\leq j \leq k$,
\item $n+1$ divides $b$.
\end{enumerate}

If $p$ is 1089, the work of Kendrick \cite{kendrick_1} 
shows that any node of the Young graph has the form $[0,0]$, $[0,n-1]$, $[n-1,0]$, or $[n-1,n-1]$, 
which establishes $(2)\Longrightarrow(3)$.
 $(3)\Longrightarrow(4)$ is easily established by Equation \ref{fund} since
$d_0=\frac{bc_1}{n^2-1}=\frac{b(n-1)}{n^2-1}=\frac{b}{n+1}$.
Theorem 6 in \cite{holt} proves $(4)\Longrightarrow(1)$.
We leave whether or not $(1)\Longrightarrow(2)$ holds as an open question. 
We note that proving this equivalence would both determine all symmetric palintiples
and further characterize all 1089 Young graphs. 

\subsection{Young Graph Isomorphism Classes of Derived Palintiples}

As we have seen, the carries of a palintiple can themselves be the digits of a lower-base palintiple.
If we elevate our perspective to entire palintiple families 
(such as Hoey palintiples), an abundance of questions present themselves:\\

\textit{Is it possible to ``derive'' new Young graphs from old using the old edges as nodes?}\\

\textit{Can we construct entire Young graph isomorphism classes from old?}\\

\textit{How are derived palintiple families related to Young graph isomorphism classes?}\\

Although we will not provide any complete answers to these questions, we will explore some suggestive examples
which, as we shall see, give rise to other questions.

Considering $(3,14)$-palintiples, their nontrivial carries are $(2,3)$-palintiple digits,
and every $(2,3)$-palintiple is the nontrivial carry sequence of some $(3,14)$-palintiple.
In other words, the Young graph describing $(3,14)$-palintiple structure can be ``derived'' from the Young graph
describing $(2,3)$-palintiple structure.
The figure below compares the Young graphs $Y(2,3)$ and $Y(3,14)$, where the ``digit-edges'' 
of the former become the ``carry-nodes'' of the latter.
(We note that our Young graph representation reverses the order of the digit-pairs associated with the edges
since the formulation of palintiples used in this article 
involves finding the number that is obtained \textit{after} multiplying by $n$.)

\includegraphics[width=5.5cm]{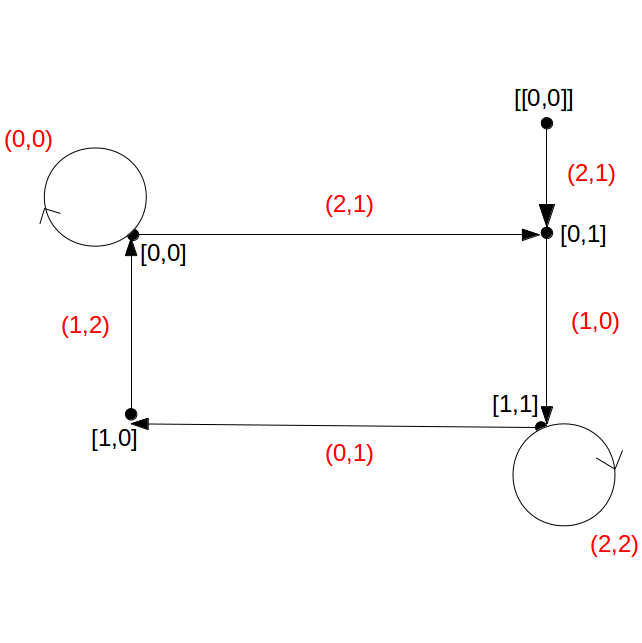}
\includegraphics[width=6.6cm]{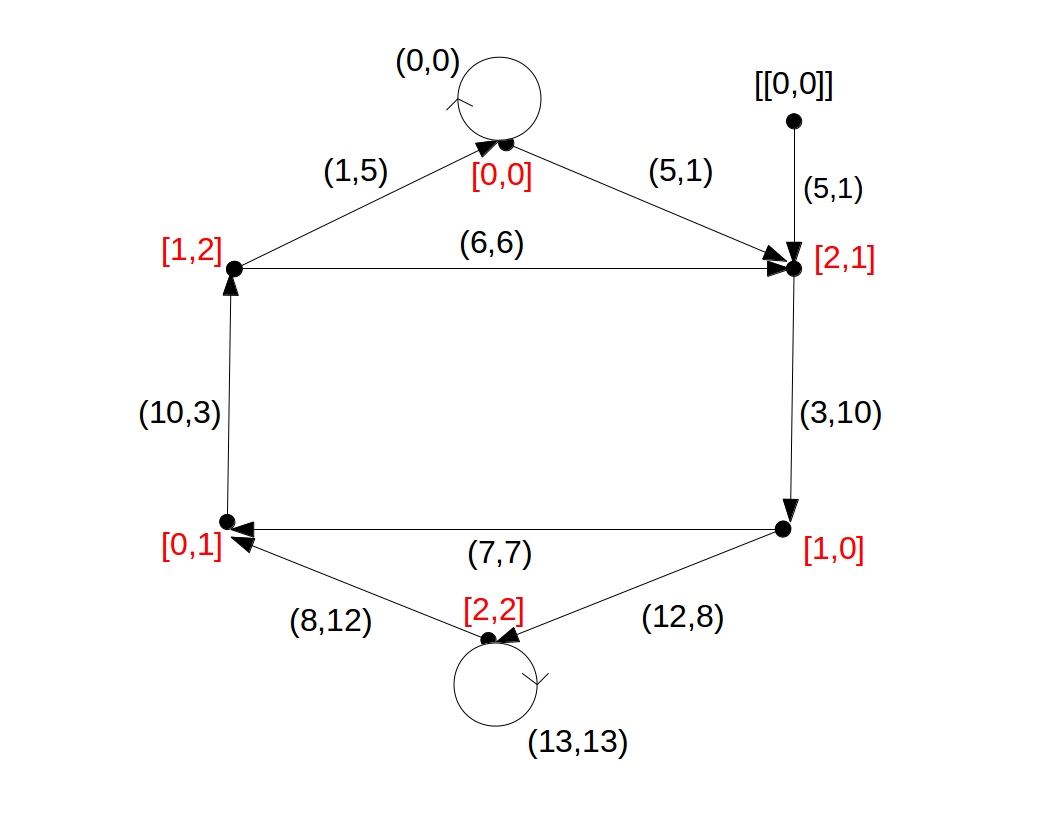}

We point out that the kind of correspondence between $(2,3)$ and $(3,14)$-palintiples does not always exist. 
In particular, an $(\hat{n},\hat{b})$-palintiple 
constructed from an $(n,b)$-palintiple does not always guarantee that the carries of any
$(\hat{n},\hat{b})$-palintiple will also be an $(n,b)$-palintiple. 
For instance, the $(9,107)$-palintiple $(12,40,1)_{107}$ has carries $(3,1,0)$ whose nontrivial elements
are the digits of the $(2,5)$-palintiple $(3,1)_5$ as seen in an earlier example. 
However, the $(9,107)$-palintiple $(24,80,2)_{107}$ has carries $(6,2,0)$  which are not
the digits of a $(2,5)$-palintiple. 

On the other hand, the family of $(5,39)$-palintiples can be constructed from $(2,5)$-palintiples. 
Consider the $(5,39)$-palintiple $(8,29,1)_{39}$ with carries $(3,1,0)$ whose nontrivial elements
are again the digits of the $(2,5)$-palintiple $(3,1)_5$. The nontrivial carries of any $(5,39)$-palintiple
are the digits of a $(2,5)$-palintiple and every $(2,5)$-palintiple is a nontrivial carry sequence of a
$(5,39)$-palintiple. 

This is all to say that, in general, the correspondence between derived palintiples and their  
palintiple carries can break down when $\hat{n} \neq b$. We therefore pose the question:\\

\textit{Suppose an $(\hat{n},\hat{b})$-palintiple can be 
derived from an $(n,b)$-palintiple. Under what conditions is it guaranteed that the carries of any
$(\hat{n},\hat{b})$-palintiple will also be an $(n,b)$-palintiple? Is $\hat{n}=b$ such a condition?}\\

Considering Hoey palintiples, it appears that not only
$Y(3,14)$, but also $Y(3,22)$, and in general $Y(3,6+8\alpha)$
for all $\alpha \geq 1$ (see the example in Section \ref{hoey_example}), can be constructed from $Y(2,3)$.
Moreover, it appears, using Kendrick's data \cite{kendrick_2}, 
that every $Y(3,6+8\alpha)$ is isomorphic to $Y(3,14)$. 
In fact, not surprisingly, for every collection of 1089 $(n,b)$-palintiples we have checked,
the Young graph of its corresponding Hoey $(b,\hat{b})$-palintiples is isomorphic to $Y(3,14)$.
In this way, the isomorphism class determined by the 1089 graph, $[Y(9,10)]$,
in a sense ``generates'' the isomorphism class $[Y(3,14)]$.

We note that not every element of $[Y(3,14)]$ is the Young graph of Hoey palintiples 
as Young graphs of $\rho$-Hoey palintiples also seem to be isomorphic to $Y(3,14)$.   
Furthermore, $[Y(3,14)]$ contains elements which are neither Young graphs of Hoey nor $\rho$-Hoey palintiples.
The $(9,14)$-palintiple $(11, 9, 1, 4, 1)_9$ with carries $(2,1,6,7,0)$ demonstrates this. 
\footnote{Although $(2,1,6,7)_b$ is neither a palintiple nor the reversal of a palintiple 
in any base $b$, the digits do give us two base-9 permutiples: 
$(6, 7, 2, 1)_9=4 \cdot (1, 6, 2, 7)_9$ and $(7, 2, 1, 6)_9=4 \cdot (1, 7, 2, 6)_9$.}
These observations lead us to ask:
\\\\
\textit{Are Young graphs of Hoey and $\rho$-Hoey palintiples always isomorphic to $Y(3,14)$?}\\\\
\textit{Are there any special properties of elements of $[Y(3,14)]$ 
which generate palintiples whose carries are not palintiple digits?}\\

Young graphs of Sutcliffe and $\rho$-Sutcliffe palintiples
derived from shifted-symmetric palintiples whose Young graph is isomorphic to 
$K_2$, $K_3$, and $K_4$, all appear to be isomorphic to $Y(7,11)$. 
Of course, considering larger values of $m$ and checking more cases may very well reveal other isomorphism classes.
We therefore ask the following:
\\\\
\textit{Are Young graphs of Sutcliffe and $\rho$-Sutcliffe palintiples always isomorphic to $Y(7,11)$?}\\

Additionally, for all cases we have checked, Young graphs of Pudwell and $\rho$-Pudwell palintiples 
derived from shifted-symmetric palintiples whose Young graph is isomorphic to 
$K_2$, $K_3$, and $K_4$, all seem to be isomorphic to $Y(5,8)$.
Thus:
\\\\
\textit{Are Young graphs of Pudwell and $\rho$-Pudwell palintiples always isomorphic to $Y(5,8)$?}\\

It is not entirely unexpected that Young graphs of Hoey, Sutcliffe, and Pudwell palintiples should be isomorphic
to Young graphs of their respective $\rho$-derived counterparts. 
On the other hand, it is not entirely obvious that this should always hold. 
In all cases considered so far, it seems to be true.  
\\\\
\textit{Are Young graphs of derived palintiples always isomorphic to their $\rho$-derived counterparts?}\\

Finally, Young graphs of $(\hat{n},\hat{b})$-palintiples derived from $(n,b)$-palintiples for which $\hat{n}\neq b$
leave cases which have hardly yet been explored. We leave the reader to ponder the example of 
$(9,107)$-palintiples considered earlier whose Young graph is isomorphic to $Y(25,59)$. 
These palintiples are in some sense ``partially'' derived from $(2,5)$-palintiples.
We suspect that these nodes might make up a subgraph, $G$, which is isomorphic to $Y(7,11)$.
The reader is likely to have noticed that other carries of $(9,107)$-palintiples are sometimes doubles of $(2,5)$-palintiples.
Thus, $Y(9,107)$ might contain another subgraph, $G'$, which is also isomorphic to $Y(7,11)$, 
but with nodes double those of $G$. Additional structure which may exist between these possible subgraphs is a matter
of further inquiry and we leave these and other such questions to the inquisitive reader.  

%


\begin{thebibliography}{}

\bibitem{hoey_1}
D. J. Hoey. Palintiples. 
Available from OEIS Foundation Inc. (2013), \textit{The On-Line Encyclopedia of Integer Sequences}
 at \url{https://oeis.org/A008919/a008919.txt}, accessed February 11th, 2016.

\bibitem{hoey_2}
D. J. Hoey. Email correspondence regarding a ``Bizarre problem in number theory.''
Available at: \url{http://keithlynch.net/DanHoey/95/1312238}, accessed February 11th, 2016.

\bibitem{holt}
B. V. Holt. Some general results and open questions on palintiple numbers, {\it Integers} {\bf 14} (2014), \#A42. 

\bibitem{kendrick_1}
L. H. Kendrick. Young graphs: 1089 et al., {\it J. Integer Seq.} {\bf 18} (2015), Article 15.9.7.

\bibitem{kendrick_2}
L. H. Kendrick. Data for Young graphs and their isomorphism classes. Available at: 
\url{https://sites.google.com/site/younggraphs/home}, accessed February 11th, 2016.

\bibitem{poly}
J. Konvalina and V. Matache. Palindrome-polynomials with roots on the unit circle, 
{\it C. R. Math. Acad. Sci. Soc. R. Can.} {\bf 26}(2) (2004), 39-44.

\bibitem{pudwell}
L. Pudwell. Digit reversal without apology, {\it Math. Mag.} {\bf 80} (2007), 129-132.

\bibitem{sloane}
N. J. A. Sloane. 2178 and all that, {\it Fibonacci Quart.} {\bf 52} (2014), 99-120.

\bibitem{sutcliffe}
A. Sutcliffe. Integers that are multiplied when their digits are reversed, {\it Math. Mag.} {\bf 39} (1966), 282-287.

\bibitem{young_1}
A. L. Young. $k$-reverse multiples, {\it Fibonacci Quart.} {\bf 30} (1992), 126-132.

\bibitem{young_2}
A. L. Young. Trees for $k$-reverse multiples, {\it Fibonacci Quart.} {\bf 30} (1992), 166-174.

\end{thebibliography}
\end{document}